\documentclass[12pt,twoside]{article}
\usepackage{amsfonts,amssymb,amsmath,amsxtra} 
\usepackage[T1]{fontenc}
\usepackage{hyperref}     
\usepackage{graphicx,color}
\usepackage{geometry, indentfirst}

\usepackage[utf8]{inputenc}

\newtheorem{theorem}{Theorem}[section]
\newtheorem{lemma}[theorem]{Lemma}
\newtheorem{proposition}[theorem]{Proposition}

\newtheorem{remark}[theorem]{Remark}
\newtheorem{definition}[theorem]{Definition}
\newenvironment{proof}[1][Proof]{\noindent\textbf{#1.} }{\hfill \rule{0.5em}{0.5em}}
\newcommand {\qed}{\hfill $\blacksquare$ \vspace{0.5cm}}

\renewcommand{\qed}{\hfill\rule{2mm}{2mm}}

\begin{document}
\title{A nonhomogeneous and critical Kirchhoff-Schr\"odinger type equation in $\mathbb R^4$ involving vanishing potentials}

\author{Francisco S. B. Albuquerque\\
Universidade Estadual da Para\'iba\\
Departamento de Matem\'atica\\
CEP: 58700-070, Campina Grande - PB, Brazil\\
\textsf{fsiberio@cct.uepb.edu.br}\\
\\
Marcelo C. Ferreira\footnote{Corresponding author}\\
Universidade Federal de Campina Grande\\
Unidade Acad\^emica de Matem\'atica\\
CEP: 58429-900, Campina Grande - PB, Brazil\\
\textsf{marcelo@mat.ufcg.edu.br}}  

\date{}

\maketitle

\begin{abstract}
  In this paper we establish the existence and multiplicity of  weak solutions to a  Kirchhoff-Schr\"odinger type problem in $\mathbb R^4$ involving a critical nonlinearity and a suitable small  perturbation. The fact that Sobolev exponent is $2^*=4$ in four dimensions, causes difficulties to treat our study from a variational viewpoint. Some tools we used in this paper are the Mountain-Pass and Ekeland's Theorems and the Lions' Concentration Compactness Principle.
\end{abstract}




\section{Introduction and main results}

In this paper we establish the existence of mountain-pass and negative energies type solutions to the following nonhomogeneous Kirchhoff-Schr\"odinger type problem:
\begin{equation*}
     \begin{cases}
   - \left( a+b \displaystyle\int_{\mathbb R^4} |\nabla u|^2 \, \text{d}x \right) \Delta u+V(x)u= \mu K(x)|u|^{q-2}u+ u^3+h(x), \ x \in \mathbb R^4,  \\
   u \in D^{1,2}(\mathbb R^4), \hfill (P_{\mu}) 
\end{cases}
\end{equation*}
where $a,b > 0$ are constants, $\mu>0$ is a parameter, $q\in (2,4)$, $h\in L^{\frac{4}{3}}(\mathbb R^4)$, and the weights $V,\ K\colon \mathbb R^4\to \mathbb R^+$ satisfy the hypotheses
	
\begin{enumerate}
 \item[$(K)$] $K  \in L^\infty(\mathbb R^4)$ and for any sequence of Borel sets  $(A_n)$ in  $\mathcal{P}( \mathbb R^4)$  such that $|A_n|\leq R$, for all $n$ and some $R>0$, it is fulfilled
 $$
 \lim_{r\rightarrow+\infty}\int_{A_n\cap B_r^{c}(0)}K(x)\, \text{d}x=0,\quad \text{uniformly in}\ n\in\mathbb{N},
 $$
 where $|\cdot|$ means the Lebesgue measure in $\mathbb R^4$;
 \item[$(VK)$] The condition $$\frac{K}{V}\in L^{\infty}(\mathbb R^4)$$
 occurs. 
\end{enumerate}

Simple examples of $V$ and $K$ satisfying $(K)$ and $(VK)$ are given by
$$
V(x)=\frac{1}{1+|x|^{\alpha}}\quad\text{and}\quad K(x)=\frac{1}{1+|x|^{\beta}},
$$
with $\beta>\alpha>4$. The potentials above belong to a class entitled vanishing at infinity (or zero mass case). After the work by  Ambrosetti, Felli and Malchiodi in \cite{afm}, lots of types of stationary nonlinear Schrödinger equations involving vanishing potentials at infinity have been studied in $\mathbb R^N$ $(N\geq 2)$ and, in the vast list of references in this aspect, we may cite \cite{AlvesSouto,NW,SWW1,SWW2} and the references therein.

We point out that the hypotheses $(K)$ and $(VK)$ were  introduced by Alves and Souto in \cite{AlvesSouto} and the authors observed they are more general than that ones considered earlier by Ambrosetti, Felli and Malchiodi in \cite{afm} in order to get compactness embedding from $E$ to $L_K^p(\mathbb{R}^4)$ (see the definitions below).  For example, if $B_n$ is a disjoint sequence of open balls in $\mathbb R^4$ centered in $x_n=(n,0,0,0)$ and $f\colon\mathbb R^4\to\mathbb R^+$ is defined as 
$$
f(x)=0, ~\forall x\in \mathbb R^4\setminus \bigcup_{n=1}^\infty B_n, ~~ f(x_n)=1~~\text{and}~~\int_{B_n} f(x)\, \text{d}x = \frac{1}{2^n},
$$ 
then, a straightfoward calculation shows 
$$
K(x)=V(x)=f(x)+\frac{1}{\ln(2+|x|)}.
$$
satisfy the hypotheses $(K)$ and $(VK)$, but $K$ does not vanish at infinity.
\medskip

Moreover, we will also assume the following assumption:
\medskip
 
\noindent $(S)$ The coefficient $b$ satisfies $b > 1/S^2$, where $S$ is the best Sobolev constant for the embedding of the Sobolev space $D^{1,2}\left( \mathbb R^4\right)$ into $L^4\left( \mathbb R^4\right)$, that is, 
$$
         S=\inf_{\stackrel{u\in D^{1,2}\left( \mathbb R^4\right)}{u\neq 0}} \frac{\displaystyle\int_{\mathbb R^4} \left|\nabla u\right|^2\,\text{d}x}{\left(\displaystyle\int_{\mathbb R^4} u^4\,\text{d}x\right)^{\frac{1}{2}}}.
$$

A problem as $(P_{\mu})$ is called \emph{nonlocal} due to the presence of the term  $\left(\displaystyle\int_{\mathbb R^4} |\nabla u|^2 \, \text{d}x\right)\Delta u$ in its formulation which implies that the equation in $(P_{\mu})$ is no longer a pointwise identity. As we will see later, this phenomenon causes some mathematical difficulties and consequently motivates the study of such a class of problems from the mathematical viewpoint. In this sense, we would like to notice that condition $(S)$ imposes our results are rather different from the most in literature, since they are not extensions of results obtained for local Schr\"odinger problems to the nonlocal case. They are purely nonlocal.

Regarding to problem $(P_{\mu})$, there are a considerable number of physical appeals. For instance, in $(P_{\mu})$ if we set $V(x)=0$, and replace $\mu K(x)|u|^{q-2}u+ u^3+h(x)$ and $\mathbb R^4$ by $f(x,u)$ and $\Omega \subset \mathbb R^N$  a bounded domain, respectively, it reduces to the following Dirichlet problem of Kirchhoff type:
\begin{equation*}
    \begin{cases}
           - \left( a+b \displaystyle\int_{\Omega} |\nabla u|^2 \, \text{d}x \right) \Delta u= f(x,u),& \, x \in \Omega, \\
           \phantom{- \left( a+b \displaystyle\int_{\Omega} |\nabla u|^2 \, dx \right) \Delta}u=0,& \, x \in \partial \Omega,        
      \end{cases}
\end{equation*}
which is related to the stationary analogue of the evolution problem
\begin{equation} \label{E:kirchhoff}
\begin{cases}
u_{tt} - \left( a+b \displaystyle\int_{\Omega} |\nabla u|^2 \, \text{d}x \right) \Delta u = f(x,u), &\, (x,t) \in \Omega\times (0,T), \\
\phantom{u_{tt} - \left( a+b \displaystyle\int_{\Omega} |\nabla u|^2 \, \text{d}x \right) \Delta u}u = 0, &\, (x,t) \in \partial\Omega\times (0,T), \\
u(x,0) = u_0(x) \quad \text{and}  \quad u_t(x,0) = u_1(x), &\ x\in \Omega.
\end{cases}
\end{equation}
Such a hyperbolic equation is a general version of the Kirchhoff equation
\begin{equation*} 
   \varrho \frac{\partial^2 u }{\partial t^2}-\left(\frac{P_0}{s}+\frac{E}{2L}\int_0^L \left| \frac{\partial u}{\partial x}\right|^{2} \, \text{d}x \right) \frac{\partial^2 u}{\partial x^2}=0, \quad (x,t) \in(0,L)\times (0,T),
\end{equation*}
which has came to light at Kirchhoff \cite{Kirchhoff}, in 1883, as an extension of the classical well-known D'Alembert wave equation for free vibrations of elastic strings. The Kirchhoff's model takes into account the effects of changes in the length of the string during the vibrations. The parameters in the above equation have the following meanings: $L$ is the length of the string, $s$ is the area of cross-section, $E$ is the Young modulus of the material, $\varrho$ is the mass density and $P_0$ is the initial tension. We recall that nonlocal problems also appear in other fields, for instance, biological processes where the function $u$ describes a distribution which depends on the average of itself (for example, population density), see for instance \cite{AC,ACF} and its references.

Some early research on Kirchhoff equations can be found in the seminal works \cite{Bernstein,Pohozaev}. However, the problem \eqref{E:kirchhoff} received great
attention of a lot of researchers only after Lions \cite{Lions} proposed an abstract framework for it, more precisely, a functional analysis approach was proposed to study it (see \cite{AC,ACF,AS,AP,CCS,PZ,ZP}). Recently, many approaches involving variational and topological methods have been used in a straightforward and effective way in order to get solutions in a lot of works (see  \cite{AlvesMFigueiredo,FS,MZ,Naimen0,Naimen,Nie,NW} and the references therein). The studies of Kirchhoff type equations have also already been extended to the case involving the $p$-Laplacian, for example \cite{CF,D1,LZ} and so on. Sometimes, the nonlocal term appears in generic form $m\left(\displaystyle\int_{\Omega} |\nabla u|^2 \, \text{d}x \right)$, where $m\colon\mathbb{R}_+\rightarrow\mathbb{R}_+$ is a continuous function that must satisfy some appropriate conditions (amongst them, monotonicity or boundedness below by a positive constant), which the typical example is given by the model considered in the original Kirchhoff equation \eqref{E:kirchhoff}. In \cite{ACF,MFigueiredo}, for example, the authors have used comparison between minimax levels of energy to show that the solution of the truncated problem, that is, an auxiliary problem obtained by a truncation on function $m$, is a solution of the original problem.  

Specifically in relation to Kirchhoff-Schr\"odinger type problems such as $(P_{\mu})$, it get so many attention, mainly in unbounded domains, due to the lack of compactness of the Sobolev's embeddings, which makes the study of the problem more delicate, interesting and challenging. In order to overcome this trouble and to recover the compactness of the Sobolev's embeddings, some authors studied their problems in a subspace consisting of radially symmetric functions. This was used in \cite{CL} for example, where Chen and Li  established multiple solutions for nonhomogenous Schrödinger-Kirchhoff problem 
\begin{equation}\label{patos}
- \left( a+b \displaystyle\int_{\mathbb{R}^{N}} |\nabla u|^2 \, \text{d}x \right) \Delta u+V(x)u= f(x,u)+h(x), \quad x\in\mathbb{R}^{N},
\end{equation}
by using Ekeland's variational principle and Mountain-Pass Theorem, with the subcritical nonlinearity $f$ satisfying the Ambrosetti-Rabinowitz condition, i.e. there exists $\theta>4$ such that 
$$
0<\theta F(x,t)=\theta \int_0^{t}f(x,s)\,\textrm{d}s,\, \forall x\in\mathbb R^N,\, t\in\mathbb R\setminus\{0\}.$$ 
For $h=0$, studies still using the subspace of radially symmetric functions can be seen in \cite{Nie,NW}.  A study for nonhomogenous Schrödinger-Kirchhoff problem  \eqref{patos} with the general nonlinearity $F$ satisfying super-quartic condition can be found in Cheng \cite{Cheng}. 

The role played by the nonhomogeneous term $h$ in producing multiple solutions is crucial in our analysis. For this reason, the study of existence of multiple solutions for nonhomogeneous elliptic equations with subcritical and critical growth in bounded and unbounded euclidean domains have received much attention in recent years (see \cite{AT,Rabi,RS,Tarantello}).

Motivated by the above works, the aim of the present paper is to continue the study of the critical nonlocal elliptic equations. To the best of our knowledge, in current literature, there are no results on the problem $(P_{\mu})$ (neither in the nonhomogeneuos nor in the homogeneous case, i.e. $h=0$).  
We emphasize that in four dimensions  $2^{\ast}=4$ is the critical Sobolev exponent of the embeddings $H^1\left(\mathbb R^4\right)\hookrightarrow L^p(\mathbb R^4)$.  This causes a tie between the growth of the nonlocal term and critical nonlinearity and, consequently, as in \cite{Nie}, one of the difficulties is investigating the boundedness of the Palais-Smale sequences, which cannot be  proved directly as usual.  Another difficulty is related to the mountain-pass geometry. We desire that the class of the corresponding energies satisfy that geometry, but we cannot  prove it in the usual way. In this sense, the assumption $(S)$ plays a fundamental role to reach our above goals. Surprisingly, $(S)$ also furnishes a welcome compactness result for our energies.   

We need to introduce some notations. From now on, we write $\displaystyle\int u$ instead of $\displaystyle\int_{\mathbb R^4}u(x) \, \text{d}x$ and we use $C,C_0,C_1,C_2,\ldots$ to denote (possibly different) positive constants. We denote by $B_R(x)\subset\mathbb{R}^{4}$ the open ball centered at $x\in\mathbb{R}^{4}$ with radius $R>0$ and $B^{c}_R(x):=\mathbb{R}^{4}\setminus B_R(x)$. The symbols $o_\varepsilon(1)$ and $o_n(1)$ will denote quantities that converge to zero, as $\varepsilon\to 0$ and $n\to \infty$ respectively. Also, we denote the weak convergence in $X$ by ``$\rightharpoonup$'' and the strong convergence by ``$\to$''.  Besides,  from the assumptions on $V$, the quantity
$$ 
\|u\|^2=\int\left(|\nabla u|^2+V(x)u^2\right),
$$ 
defines a norm over 
$$
E=\left\{u\in D^{1,2}(\mathbb R^4):\displaystyle\int V(x)u^2<\infty\right\}
$$ 
(our work space) such that  $E$ is Hilbert, $E$ is continuously imersed in $D^{1,2}(\mathbb R^4)$ and $L^4(\mathbb R^4)$. Furthermore, under the hypotheses $(K)$ and $(VK)$ in the study of Alves and Souto (see \cite[Proposition 2.1]{AlvesSouto}), we know that $E$ is compactly embedded into the weighted Lebesgue space
$$
L_K^p(\mathbb{R}^4):=\left\{u\colon\mathbb{R}^4\rightarrow\mathbb{R} : u\ \text{is measurable and} \int K(x)|u|^p<\infty\right\},
$$
equipped with the norm
$$
|u|_{p;K}=\left(\int K(x)|u|^p\right)^{\frac{1}{p}},
$$ 
for all $2< p<4$. Also, for $u\in L^p(\mathbb R^4)$ we denote its $p$-norm with respect to the Lebesgue measure by $|u|_p$  and $E^{\ast}$ will designate the dual space of  $E$ with the usual norm $\|\cdot\|_{E^{\ast}}$.
\begin{definition}
We say that $u\colon\mathbb{R}^{4}\rightarrow\mathbb{R}$ is a weak solution of $(P_{\mu})$ if $u \in E$ and it holds the identity
$$
\left(a+b\int|\nabla u|^2\right) \int\nabla u\cdot\nabla \varphi+\int V(x)u\varphi=\mu\int K(x)|u|^{q-2}u\varphi+\int u^3\varphi+\int h\varphi,
$$
for all $\varphi\in E$.
\end{definition}

The main results of this work can be stated as follows.
\begin{theorem}\label{mth}
Assume that $(K),\ (VK)$ and $(S)$ hold. Then, there exists $\mu^*>0 $ sufficiently large such that  $(P_{\mu})$ has a positive energy  weak solution in $D^{1,2}(\mathbb R^4)$ for almost everywhere $\mu > \mu^*$, whenever $0\leq |h|_{\frac{4}{3}}$ is sufficiently small. 
\end{theorem}

The proof of Theorem \ref{mth} is based in a result presented in \cite{Jeanjean}. As we said above, mainly in order  to prove the boundedness of some Palais-Smale sequences, which cannot be proved directly in our case.

\begin{theorem}\label{mth1}
Assume that $(K),\ (VK)$ and $(S)$ hold. Then, for each $\mu>0 $, problem $(P_{\mu})$ has a negative energy weak solution in $D^{1,2}(\mathbb R^4)$, whenever $0<|h|_{\frac{4}{3}}$ is sufficiently small.
\end{theorem}

The proof of Theorem \ref{mth1} is based on Ekeland's variational principle (see \cite{Ekeland}) to prove the existence of a local minimum type solution. 
\smallskip

Theorems \ref{mth} and \ref{mth1} can be combined to give the following one:

\begin{theorem}
Assume that $(K),\ (VK)$ and $(S)$ hold. Then, there exists $\mu^*>0 $ sufficiently large such that  $(P_{\mu})$ has at least two different  weak solution in $D^{1,2}(\mathbb R^4)$ for almost everywhere $\mu > \mu^*$, whenever $0< |h|_{\frac{4}{3}}$ is sufficiently small. 
\end{theorem}

The outline of the paper is as follows: Section 2 contains the variational setting in which our problem will be treated and allow us to follow a variational approach. Section 3 is devoted to study convenient properties of some Palais-Smale sequences and of the functional $I_\mu$. The proofs of the main results are established in Section 4. 

\section{Preliminary results}

Following the line firstly introduced by Alves et al. in \cite{ACF} to solve the Kirchhoff
problem, we establish now the necessary functional framework where solutions are naturally studied by variational methods. We begin by noticing that hypotheses $(K)$ and $(VK)$ ensure $E\hookrightarrow L_K^q(\mathbb R^4)$ and, consequently, 
$$
\int K(x)|u|^q<\infty ,~\forall u\in E.
$$
This allows us to consider $I_{\mu}\colon E\to\mathbb R$, where
\begin{multline*}
I_{\mu}(u)=\frac{1}{2}\int\left(a|\nabla u|^2+V(x)u^2\right)\\+\frac{b}{4}\left(\int|\nabla u|^2\right)^2-\frac{\mu}{q}\int K(x)|u|^q-\frac{1}{4}\int u^4-\int hu.
\end{multline*}
Moreover,  it can be showed that $I_{\mu}\in C^1\left(E,\mathbb R\right)$ with derivative given by
\begin{multline*}
I'_{\mu}(u)v=\left(a+b\int|\nabla u|^2\right) \int\nabla u\cdot\nabla v+\int V(x)uv\\
-\mu\int K(x)|u|^{q-2}uv-\int u^3v-\int hv,~\forall u,v\in E.
\end{multline*}
So that, any critical point of the functional $I_{\mu}$ is a weak solution to problem $(P_{\mu})$ and conversely.

\subsection{The mountain-pass geometry}

Next two Lemmas describe the geometric structure of the functional $I_{\mu}$ required by the Mountain-Pass Theorem due to Ambrosetti and Rabinowitz in \cite{ambro-rabi}. Due to the tie in growth for the nonlocal term and critical nonlinearity, we would like to point out that assumption $(S)$ plays a important role for the proof of the second one.

\begin{lemma}\label{L:geo1}
Let $\mu >0$. Then, there exists $\delta_{\mu}>0$ such that for $h\in L^{\frac{4}{3}}(\mathbb R^4)$ with $|h|_{\frac{4}{3}}<\delta_\mu$, it holds that
$$
     I_{\mu}(u)\geq \sigma, \textrm{  for } \|u\|=\tau,
$$
for some $\sigma>0$ and  $0<\tau<1$.
\end{lemma}

\begin{proof}
The continuous embeddings $E\hookrightarrow L^q_K(\mathbb R^4)$ and $E\hookrightarrow L^4(\mathbb R^4)$,   yields
\begin{align*}
I_{\mu}(u)&\geq\frac{\min\{a,1\}}{2}\|u\|^{2}-\mu C_0\|u\|^{q}-C_1\|u\|^4-C_2|h|_{\frac{4}{3}}\|u\|\\
&=\|u\|\left(\frac{\min\{a,1\}}{2}\|u\|-\mu C_0\|u\|^{q-1}-C_1\|u\|^3-C_2|h|_{\frac{4}{3}}\right).
\end{align*}
Taking $0<\tau<1$ such that $\mu C_0\tau^{q-1}+C_1\tau^{3}\leq \dfrac{\min\{a,1\}}{4}\tau$, then for $\|u\|=\tau$ we have
$$
I_{\mu}(u)\geq \tau\left(\frac{\min\{a,1\}}{4}\tau-C_2|h|_{\frac{4}{3}}\right).
$$
Thus, for $|h|_{\frac{4}{3}}<\delta_\mu=\dfrac{\min\{a,1\}}{4C_2}\tau$ , we derive 
$$
I_{\mu}(u)\geq \sigma= \tau\left(\frac{\min\{a,1\}}{4}\tau-C_2|h|_{\frac{4}{3}}\right), \ \text{if}\ \|u\|=\tau,
$$
which concludes the proof.
\end{proof}

\begin{remark}
In the above proof, we take $C_2=S^{-1/2}$. Since $0<\tau<1$, for all $\mu>0$ we have 
$$
\delta_\mu<\frac{\min\{a,1\}}{4S^{-1/2}}.
$$
\end{remark}
\medskip

\begin{lemma} \label{L:geo2}
Let $\mu>0$ and assume $|h|_{\frac{4}{3}}<\delta_\mu$. 
Then, there exists $w\in E$ satisfying
$$
      \|w\|>1\ \text{ and }\ I(w)<0,
$$
for all $\mu>0$ sufficiently large. 
\end{lemma}

\begin{proof}
 Fix $u\in C^\infty_0( \mathbb R^4)\setminus\{0\}$. For each $t>0$, we set 
$$
   u_t(x)=u\left(\frac{x}{t}\right),\ t\in(0,\infty).
$$
A straightforward computation yields
\begin{multline*}
I_{\mu}(u_t)\leq \frac{a}{2}\left(\int|\nabla u|^2\right)t^2+\left(S^{-1/2}|h|_{\frac{4}{3}}\|u\|\right) t\\+\left(\frac{1}{2}\int V(tx) u^2-\frac{\mu}{q}\int K(tx)|u|^q+\frac{1}{4}\left[b\left(\int|\nabla u|^2\right)^2
-\int u^4\right]\right)t^4.
\end{multline*}
We notice that due to assumption $(S)$, there holds
\begin{equation}\label{E:(S)}
b\left(\int|\nabla u|^2\right)^2-\int u^4>0.
\end{equation}
Let
$$
A=\frac{a}{2}\left(\int|\nabla u|^2\right)~\text{ and }~B=\frac{\min\{a,1\}}{4}\|u\|.
$$
We fix $t_0\approx \infty$ such that
$$
At_0^2+Bt_0-t_0^4<0  \ \text{ and }\ 
\|u_{t_0}\|>1.
$$
From \eqref{E:(S)}, we can choose $\mu_0>0$ such that 
$$
-1=\frac{1}{2}\int V(t_0x)u^2-\frac{\mu_0}{q}\int K(t_0x)|u|^q+\frac{1}{4}\left[b\left(\int|\nabla u|^2\right)^2
-\int u^4\right].
$$
If  $\mu> \mu_0$, then 
$$
  \frac{1}{2}\int V(t_0x) u^2-\frac{\mu}{q}\int K(t_0x)|u|^q+\frac{1}{4}\left[b\left(\int|\nabla u|^2\right)^2
-\int u^4\right]< -1.
$$ 
We get
$$
I_{\mu}(u_{t_0})<At_0^2+Bt_0-t_0^4<0.
$$
So that, if we take $w=u_{t_0}$, then $\|w\|> 1$ and $I_{\mu}(w)<0$,  which ends the proof.
\end{proof}


Next, we summarized the last two Lemmas in the following:

\begin{proposition} \label{P:mountainpass}
Let $\mu>\mu_0$. Then, there exists $\delta_{\mu}>0$ such that  the energy $I_{\mu}$ satisfies the mountain-pass Geometry, whenever $|h|_{\frac{4}{3}}<\delta_\mu$.
\end{proposition}

\begin{remark}\label{rmk-mu}  Hereafter, for each $\mu>\mu_0$,  we will take the corresponding mountain-pass levels 
$$
      c_{\mu}=\inf_{\gamma\in\Gamma}\max_{t\in[0,1]} I_{\mu}(\gamma(t)) 
$$ 
over the same class of paths $\Gamma$, that is,
$$
      \Gamma=\left\{ \gamma \in C\left([0,1],E\right):\gamma(0)=0,\ \gamma(1)=w \right\}.
 $$
\end{remark}

\subsection{A local minimum of $I_\mu$ near the origin}

In what follows,  
$$
B_\tau=\left\{u\in E: \|u\|\leq \tau\right\},
$$
where $\tau>0$ is given by Lemma \ref{L:geo1}. Evidently, $B_\tau$ is a complete metric subspace of $E$.

\begin{proposition} \label{L:IBB}
The functional $I_{\mu}$ is bounded below in $B_\tau$. Moreover, if $h\neq  0$ and 
$$
\nu_{\mu}=\displaystyle\inf_{B_\tau} I_{\mu},
$$ 
then $\nu_{\mu}<0$.
\end{proposition}

\begin{proof}
Due to the continuous $E\hookrightarrow L^q_K(\mathbb R^4)$ and $E\hookrightarrow L^4(\mathbb R^4)$, we have
$$
\left|I_{\mu}(u)\right|\leq \max\left\{a,1\right\}\|u\|^2 +\frac{b}{4}\|u\|^4+\mu C_0\|u\|^q + C_1\|u\|^4+C_2|h|_{\frac{4}{3}}\|u\|.
$$
Then, for $\|u\|\leq \tau$,
$$
\left|I_{\mu}(u)\right|\leq \max\left\{a,1\right\}\tau^2+\frac{b}{4}\tau^4+\mu C_0\tau^q+C_1\tau^4+C_2|h|_{\frac{4}{3}}\tau:=C,
$$
which shows
$$
I_{\mu}(u)\geq -C,~\forall u\in B_\tau,
$$
as desired. Now, let  $$\nu_{\mu}=\displaystyle\inf_{B_\tau} I_{\mu},
$$ and fix $u\in E\setminus\{0\}$ such that  $\displaystyle\int hu>0$. Given $t>0$, there holds
\begin{align*}
I_{\mu}(tu) &=\frac{1}{2}\left(\int \left(a|\nabla u|^2+V(x)u^2\right)\right)t^2\\
&+\frac{1}{4}\left[b\left(\int|\nabla u|^2\right)^{2}-\int u^4\right]t^4
-\frac{\mu}{q}\left(\int K(x)|u|^q\right) t^q-\left(\int hu\right)t,
\end{align*}
showing $I_{\mu}(tu)<0$ for sufficiently small values of $t$. Since $tu\in B_\tau$ for these values, it is also true that
$$
\nu_{\mu}\leq I_{\mu}(tu),
$$
and we have achieved the proof.
\end{proof}

\begin{remark}
It will be proved in Section 4 that $\nu_\mu$ is in fact a local minimum value to $I_\mu$.
\end{remark}

\section{On Palais-Smale sequences}

Firstly we recall that $(u_n)$ in $E$ is a Palais-Smale sequence at level $d\in\mathbb{R}$ (briefly $(PS)_d$) for the functional $I_{\mu}$  if 
$$
I_{\mu}(u_n)\rightarrow d\ \text{in}\ \mathbb R\ \ \text{and}\ \ I'_{\mu}(u_n)\rightarrow 0\ \text{in}\ E^\ast\ \text{as}\ n\to+\infty.
$$

\subsection{The boundedness of some $(PS)$ sequences}

In order to prove the boundedness of Palais-Smale sequences at the mountain-pass level for $I_{\mu}$, we will use the following result due to Jeanjean \cite{Jeanjean}. This is a part really necessary in our arguments, since we cannot  prove the boundedness directly as usual.

\begin{lemma} [Jeanjean, \cite{Jeanjean}] \label{P:Jeanjean}
  Let $(X,\|\cdot\|)$ be a Banach space, $J\subset \mathbb R_+$ an interval and $(\varphi_\mu)$ be a family of $C^1$ functionals on $X$ of the form 
  $$
      \varphi_\mu(u)=A(u)-\mu B(u),\ \mu \in J,
  $$    
where $B(u)\geq 0, \, \forall u \in X$, and such that 
  $$
       A(u)\to\infty\ \text{ or }\ B(u)\to\infty,\ \text{as } \|u\|\to \infty.
  $$     
  If there exist two
points $v_1,v_2 \in X$ such that setting 
  $$
     \Gamma=\left\{ \gamma \in C([0,1],X):\gamma(0)=v_1,\ \gamma(1)=v_2 \right\},
  $$
 for all $\mu \in J$ there hold
  $$
     \beta_\mu:=\inf_{\gamma \in \Gamma} \max_{t\in [0,1]} \varphi_\mu(\gamma(t))> \max \left\{ \varphi_\mu(v_1), \varphi_\mu(v_2)\right\},
  $$
  then, for almost every $\mu \in J$, there is a bounded $(PS)_{\beta_\mu}$ sequence $(u_n)$ for $\varphi_\mu$ in $X$. 
\end{lemma}

The application of Lemma \ref{P:Jeanjean} to functional $I_{\mu}$ yields bounded  Palais-Smale sequences at mountain-pass level $c_\mu$ for large values of $\mu$. Here, once more, we should underline the role played by assumption $(S)$.

\begin{proposition}\label{L:psb}
Let $\mu^*=\mu_0>0$  given in Lemma \ref{L:geo2} and $\mu>\mu^*$. If  $|h|_{\frac{4}{3}}<\delta_{\mu}$ (see  Lemma \ref{L:geo1}),  there exists (except for a zero measure set of $\mu's$) a bounded Palais-Smale sequence for $I_{\mu}$ at level 

$$
    c_{\mu}:=\inf_{\gamma \in \Gamma} \max_{t\in [0,1]} I_{\mu}(\gamma(t)),
$$
where
$$
     \Gamma=\left\{ \gamma \in C\left([0,1],E\right):\gamma(0)=0,\ \gamma(1)=w \right\},
$$
with $w$  given by Lemma \ref{L:geo2}.
\end{proposition}

\begin{proof} 
Setting
$$
   A(u):=\frac{1}{2}\int\left(a|\nabla u|^2+ V(x)u^2\right)+\frac{1}{4}\left[b\left(\int|\nabla u|^2\right)^2-\int u^4\right]-\int hu
$$
and
$$
   B(u):=\frac{1}{q}\int K(x)|u|^q,
$$
we can consider
$$
  I_{\mu}(u)=A(u)-\mu B(u),~ \mu> \mu^*.
$$
Due to \eqref{E:(S)}, we derive
$$
A(u)\geq \frac{\min\{a,1\}}{2}\|u\|^2-C_2|h|_{\frac{4}{3}}\|u\|.
$$
Thus 
$$
A(u)\to  \infty,\ \text{as } \|u\| \to \infty.
$$
Since
$$
    c_{\mu}> \max \left\{ I_{\mu}(0), I_{\mu}(w)\right\}, ~ \mu> \mu^*,
$$
the proof follows from Lemma \ref{P:Jeanjean}.
\end{proof}

\subsection{A compactness result for $I_\mu$}

Next two Propositions provide a compactness type result for functional $I_\mu$. This fact is reached combining Lions' Second Concentration Compactness Lemma (cf. \cite[Lemma 2.1]{pLions}) and a Hardy-type inequality in the study of Alves and Souto (cf \cite[Proposition 2.1]{AlvesSouto}).

\begin{proposition}\label{L:cl4} \label{L:convL4}
Let $(u_n)$ be a bounded $(PS)_d$ sequence for $I_\mu$. Then, there exist $u_0\in L^4(\mathbb R^4)$ and a subsequence of $(u_n)$, also denoted by $(u_n)$, such that 
 $$
 u_n\to u_0\text{ in }L^4(\mathbb R^4), \text{ as }n\to \infty.
 $$
 \end{proposition}

\begin{proof}
Since $(u_n)$ is bounded in $E$ and $E$ is reflexive, there exist $u_0\in E$ and a subsequence of $(u_n)$, which we also denote by $(u_n)$, such that 
$$
u_n\rightharpoonup u_0~\text{in}~E,~\text{as}~n\to \infty.
$$
Then,  
$$
u_n(x)\to u_0(x),~\text{for a.e.}~ x\in \mathbb R^4,~\text{as}~n\to\infty,
$$
and, since $E\hookrightarrow L^4(\mathbb R^4)$, as well $(u_n)$ is bounded in $L^4(\mathbb R^4)$. 
Thus, by Brezis-Lieb Lemma (cf. \cite[Lemma 1.32]{Willem}), it is sufficient to show that 
$$
|u_n|_4\to |u_0|_4,~\text{as}~n\to \infty.
$$
By using  Lions' Second Concentration Compactness Lemma, there exist at most a countable set $\mathcal{I}$,  $\{x_{k}\}_{k \in \mathcal{I}} \subset \mathbb R^4$ and $ \{\eta_{k}\}_{k \in \mathcal{I}},\ \{\nu_{k}\}_{k \in \mathcal{I}} \subset (0,\infty)$ such that
$$ |\nabla u_n|^2 \, dx \rightharpoonup \eta \ge |\nabla u_0|^2 \: dx + \displaystyle \sum_{k \in \mathcal{I}} \eta_{k}\delta_{x_{k}},$$
$$ u_n^4\,dx \rightharpoonup\nu = u_0^4 \: dx + \displaystyle \sum_{k \in \mathcal{I}} \nu_{k}\delta_{x_{k}},$$
\begin{equation} \label{E:LCC}
\eta_{k} \ge S\nu_{k}^{\frac{1}{2}} \quad (k \in \mathcal{I}).
\end{equation}\\
Our task now is to show that $\mathcal{I} = \emptyset$. By contradiction, assume that $\mathcal{I} \neq \emptyset$. For each  $k \in \mathcal{I}$ and $\varepsilon > 0$, we consider a smooth function $\phi = \phi_{k,\varepsilon}\colon \mathbb{R}^4 \to \mathbb{R}$ such that 
$$\left\{
   \begin{array}{cc}
   \phi = 1, & \quad \textrm{in} \: B_{\varepsilon}(x_k), \\
   \phi = 0, & \quad \ \textrm{in} \: B^{c}_{2\varepsilon}(x_k), \\
   0 \le \phi \le 1, &  \ \ \ \ \ \ \ \textrm{in the remaining case}, \\
   |\nabla \phi| \le \frac{2}{\varepsilon},& \! \!\!  \! \textrm{in} \: \mathbb R^4.
   \end{array}
   \right. $$
Noticing that $I'_{\mu}(u_n)(u_n\phi) \rightarrow 0$ in $E^{\ast}$, we have
\begin{equation} \label{E:FALCC}
\lim_n \left[ \left(a+b\int |\nabla u_n|^2\right)\int |\nabla u_n|^2\phi-\int u_n^4\phi\right]+ o_\varepsilon(1) =0,
\end{equation}
where
\begin{align*}
    o_\varepsilon(1) = & \lim_n\left[  \int_{B_{2\varepsilon}(x_k)} V(x)u_n^2\phi- \int_{B_{2\varepsilon}(x_k)} hu_n\phi-\mu\int_{B_{2\varepsilon}(x_k)} K(x)|u_n|^q\phi \right.\\
    & \left. + \left( a+b\int |\nabla u_n|^2\right)\int_{B_{2\varepsilon}(x_k)} (\nabla u_n\cdot\nabla\phi)u_n\right].
\end{align*}
In fact,  combining Schwarz's inequality, H\"older's inequality and the compact embedding $E\hookrightarrow L^2_{loc}(\mathbb R^4)$, we get 
\begin{eqnarray*}
 &&\left|\displaystyle \lim_{n\to \infty} \left(a+b\int |\nabla u_n|^2 \right)\int_{B_{2\varepsilon}(x_k)}(\nabla u_n \cdot \nabla \phi) u_n \, dx  \right|\\
&&\leq  C\left(\int_{B_{2\varepsilon}(x_k)} u_0^2|\nabla \phi|^2 \: dx\right)^{1/2} \\
&& \leq  C\left(\int_{B_{2\varepsilon}(x_k)} u_0^4 \: dx\right)^{1/4}\left(\int_{ B_{2\varepsilon}(x_k)}|\nabla \phi|^4  \: dx \right)^{1/4} \\
&& \leq  C\left(\int_{ B_{2\varepsilon}(x_k)} u_0^4 \: dx \right)^{1/4},
\end{eqnarray*}
for a constant $0<C$ which does not depend on $\epsilon$, and where, in the last inequality, we use that $|\nabla \phi|\leq \dfrac{2}{\varepsilon}$. In addition, by using
analogous arguments as the previous above, we obtain
\begin{eqnarray*}
&& \left|\displaystyle \lim_{n\to \infty} \left(\int_{B_{2\varepsilon}(x_k)} V(x)u_n^2\phi- \int_{B_{2\varepsilon}(x_k)} hu_n\phi-\mu\int_{B_{2\varepsilon}(x_k)} K(x)|u_n|^q\phi \right)  \right|\\
&& \leq  \int_{B_{2\varepsilon}(x_k)} V(x)u_0^2+ \int_{B_{2\varepsilon}(x_k)} hu_0+\mu\int_{B_{2\varepsilon}(x_k)}  K(x)|u_0|^q .
\end{eqnarray*}
Thus, formula \eqref{E:FALCC} is justified. But then, by applying Lions'  Concentration Compactness Lemma on this formula, we derive
$$
      0\geq  \left[\left(a+b\int_{B_{2\varepsilon}(x_k)} \phi \: d\eta \right)\int_{ B_{2\varepsilon}(x_k)} \phi \: d\eta - \int_{ B_{2\varepsilon}(x_k)} \phi \:d\nu + o_\varepsilon(1)\right] .
$$
By passing to the limit as $\varepsilon\to 0$ and using relation \eqref{E:LCC} , we get
$$
0\geq (a+b\eta_k)\eta_k-\nu_k\geq b\eta_k^2-\nu_k\geq \nu_k(bS^2-1).
$$
Hence $$b\leq 1/S^{2},$$ which contradicts the hypothesis $(S)$. Therefore, $\mathcal{I}=\emptyset$ and the result is proved.
\end{proof}

\begin{proposition} \label{L:new}
Let $(u_n)$ be a bounded $(PS)_d$ sequence for $I_\mu$. Then, there exist $u_0\in E$ and a subsequence of $(u_n)$, also denoted by $(u_n)$, such that 
 $$
 u_n\to u_0\text{ in }E, \text{ as }n\to \infty.
 $$
 \end{proposition}

\begin{proof}
As seen in the previous Proposition, there exist $u_0\in E$ and a subsequence of $(u_n)$, which we also denote by $(u_n)$, such that 
$$
u_n\rightharpoonup u_0~\text{in}~E,~\text{as}~n\to \infty.
$$
Firstly, we notice that $(I'(u_n)-I'(u_0))(u_n-u_0)=o_n(1)$, that is,
\begin{multline}\label{bala}
    \left(a+b\int|\nabla u_n|^2\right)\int \left|\nabla(u_n-u_0)\right|^2+\int V(x)(u_n-u_0)^2 \\
    + b\left(\int |\nabla u_n|^2 -\int |\nabla u_0|^2 \right)\int \nabla u_0\cdot \nabla(u_n-u_0)\\-\mu\int K(x)(|u_n|^{q-2}u_n-|u_0|^{q-2}u_0)(u_n-u_0)\\ 
     -\int (u_n^3-u_0^3)(u_n-u_0)=o_n(1). 
\end{multline}
For what follows, we set
$$
I_n^1=\left(\int |\nabla u_n|^2 -\int |\nabla u_0|^2 \right) \int \nabla u_0\cdot \nabla(u_n-u_0), 
$$ 
$$
I_n^2=\int K(x)(|u_n|^{q-2}u_n-|u_0|^{q-2}u_0)(u_n-u_0)
$$
$$
\text{ and }~~ 
I_n^3=\int (u_n^3-u_0^3)(u_n-u_0).
$$
We claim that $I_n^1,~ I_n^2,~ I_n^3 \rightarrow 0,~ \text{as}~ n\rightarrow +\infty$. Recalling that $(u_n)$ is bounded in $D^{1,2}(\mathbb R^4)$, the first of these convergences follows immediately from the weak convergence $u_n\rightharpoonup u_0$ in $D^{1,2}(\mathbb R^4)$. Next, we shall verify the other two in the following steps.
\medskip

\noindent {\bf Step 1:}  $I_n^2=o_n(1)$. 
\medskip

By H\"older's inequality 
\begin{align*}
       |I_n^2|&\leq \int K(x)^{\frac{q-1}{q}}||u_n|^{q-2}u_n-|u_0|^{q-2}u_0|^{\frac{q}{q-1}}K(x)^{\frac{1}{q}}|u_n-u_0|\\
       &\leq \left|K^{\frac{q-1}{q}}||u_n|^{q-2}u_n-|u_0|^{q-2}u_0|\right|_{\frac{q}{q-1}}\left|K^{\frac{1}{q}}|u_n-u_0|\right|_q\\
      & = \left( \int K(x)\left||u_n|^{q-2}u_n-|u_0|^{q-2}u_0\right|^{\frac{q}{q-1}}\right)^{\frac{q-1}{q}}\left(\int K(x)|u_n-u_0|^q\right)^{\frac{1}{q}}.
\end{align*}
Since $E\hookrightarrow L^q_K(\mathbb R^4)$, we get $(u_n)$ is bounded in $L^q_K(\mathbb R^4)$ and there holds
\begin{multline*}
\left(\int K(x)||u_n|^{q-2}u_n-|u_0|^{q-2}u_0|^{\frac{q}{q-1}}\right)^{\frac{q-1}{q}}\\
\leq \left(\int K(x)\left(|u_n|^{q-1}+|u_0|^{q-1}\right)^{\frac{q}{q-1}}\right)^{\frac{q-1}{q}}\\
\leq 2\left(\int K(x)|u_n|^q+\int K(x)|u_0|^q\right)^{\frac{q-1}{q}}\leq c_2,
\end{multline*}
for some constant $c_2>0$. So that,
$$
|I_n^2|\leq c_2|u_n-u_0|_{q,K},
$$
and, by the compact imersion $E\hookrightarrow L^q_K(\mathbb R^4)$, we derive $I_n^2=o_n(1)$,  as we claimed. \bigskip

\noindent {\bf Step 3:} $I_n^3=o_n(1) $. 
\medskip

In fact,  by using H\"older's inequality
$$
      |I_n^3|\leq \int |u_n^3-u_0^3||u_n-u_0| \leq |u_n^3-u_0^3|_{\frac{4}{3}} |u_n-u_0|_4\leq \left(\int |u_n^3-u_0^3|^{\frac{4}{3}} \right)^{\frac{3}{4}}|u_n-u_0|_4. 
$$
Since $(u_n)$ is bounded in $L^4(\mathbb R^4)$, it is fulfilled
$$
\left(\int |u_n^3-u_0^3|^{\frac{4}{3}}\right)^{\frac{3}{4}} \leq 2\left( \int |u_n|^4+ \int |u_0|^4\right)^{\frac{3}{4}}\leq c_3,
$$
for some constant $c_3>0$. So that,
$$
 |I_n^3|\leq c_3|u_n-u_0|_4
$$ 
and, by Proposition \ref{L:cl4},  $I_n^3=o_n(1)$, which concludes the verification of the claim.
\medskip

From the above convergences, \eqref{bala},  it follows that
$$
     \left(a+b\int|\nabla u_n|^2\right)\int \left|\nabla(u_n-u_0)\right|^2+\int V(x)(u_n-u_0)^2=o_n(1).
$$
Thereby, we derive
$$
\|u_n-u_0\|=\int \left|\nabla(u_n-u_0)\right|^2+\int V(x)(u_n-u_0)^2\leq o_n(1),
$$
which proves that $u_n\to u_0\in E$, as $n\to \infty$.
\end{proof}

\section{Proofs of  the main results}

In this section, we will prove Theorems \ref{mth} and \ref{mth1}. 

\subsection{The proof of Theorem \ref{mth}}

Let $\mu^*>0$  given by  Proposition \ref{L:psb}. For $\mu> \mu^*$ and $|h|_{\frac{4}{3}}<\delta_\mu$, there exists (almost everywhere) a bounded Palais-Smale sequence $(u_n)$ in $E$ for $I_{\mu}$ at mountain-pass level $c_{\mu}$. By Proposition \ref{L:new}, there exists $u_0\in E$ and subsequence of $(u_n)$, which we also denote by $(u_n)$ such that 
$$
 u_n\to u_0\text{ in }E, \text{ as }n\to \infty.
$$
Since $I_\mu$ is $C^1(E,\mathbb R)$, , we derive $I_\mu(u_0)=c_\mu$ and $I'_\mu(u_0)=0$. So that, $u_0$ is a solution to $(P_{\mu})$ with positive energy. Thus, the Theorem  \ref{mth} is proved. \qed  

\subsection{The proof of Theorem \ref{mth1}}

We will prove here the minimization of $I_\mu$ constrained to $B_\tau$, by using the following result due to Ekeland (see \cite{Ekeland}):

\begin{proposition} [Ekeland's variational principle - weak form] \label{P:EVP}
Let $(X,d)$ be a complete metric space and $\Phi\colon X\to\mathbb R\cup\left\{ \infty\right\}$ be a lower semicontinuous functional, which is bounded below.  Then, for each given $\varepsilon>0$, there exists $u_\varepsilon\in X$ such that
$$
\Phi(u_\varepsilon)< \inf_X \Phi+\varepsilon
$$
and
$$
\Phi(u_\varepsilon)<\Phi(u)+\varepsilon d(u_\varepsilon,u),~u\neq u_\varepsilon.
$$
\end{proposition} 
\bigskip

In fact, applying it to $X=B_\tau$ and $\Phi=I_\mu$, in light of the Proposition \ref{L:IBB} and since $I_{\mu}$ is lower semicontinuous in $B_\tau$, using Proposition \ref{P:EVP} provides a sequence $(v_n)$ in  $B_\tau$ such that
$$
I_{\mu}(v_n)<\nu_{\mu}+\frac{1}{n}
$$
and
$$
I_{\mu}(v_n)<I_{\mu}(u)+\frac{1}{n}\|v_n-u\|,~ u\in B_\tau,~ u\neq v_n. 
$$
\medskip

The proof of Theorem \ref{mth1} is a consequence of the following Lemma.

\begin{lemma}
Let $\mu>0$. If  $0<|h|_{\frac{4}{3}}<\delta_\mu$, then $(v_n)$ is a $(PS)_{\nu_{\mu}}$ sequence for $I_{\mu}$.
\end{lemma}

\begin{proof}
Since  $I_{\mu}(v_n)\to  \nu_{\mu}<0$,
without loss of generality we can assume that $I_{\mu}(v_n)<0,~\forall n\in \mathbb N$. From Lemma \ref{L:geo1}, we have $v_n\in \mathring{B_\tau}$, where $\mathring{B_\tau}$ denotes the interior of $B$. Therefore, for $v\in E$ such that $\|v\|\leq 1$, and for any small positive value of $\theta\in\mathbb R$, we get
$$
v_n+\theta v\in \mathring{B_\tau}~\text{ and } ~ v_n+\theta v\neq v_n.
$$
But then
$$
I_{\mu}(v_n)<I_{\mu}(v_n+\theta v)+\frac{\theta}{n}.
$$
The differentiabilty of $I_{\mu}$ implies that
$$
I'_{\mu}(v_n)v\geq -\frac{1}{n}.
$$
Replacing $v$ by $-v$ we obtain
$$
I'_{\mu}(v_n)v\leq \frac{1}{n}.
$$
Thus, $\|I'_{\mu}(v_n)\|_{E^\ast}\to 0$.
\end{proof}
\bigskip

Since $(v_n)$ is a bounded $(PS)_{\nu_\mu}$ sequence for $I_\mu$, by Proposition \ref{L:new}, there exist $v_0\in E$ and subsequence of $(v_n)$, which we also denote by $(v_n)$, such that 
$$
 v_n\to v_0\text{ in }E, \text{ as }n\to \infty.
$$
We derive $I_\mu(v_0)=\nu_\mu$ and $I'_{\mu}(v_0)=0$. So that, $v_0$ is a solution to $(P_\mu)$ with negative energy. Thus, the Theorem  \ref{mth1} is proved. \qed

\end{document}